\documentclass[11pt]{amsart}
\usepackage{amssymb}
\usepackage{amsaddr}
\usepackage[margin=1in]{geometry}
\usepackage[colorlinks]{hyperref}
\usepackage{graphicx}

\newtheorem{theorem}{Theorem}[section]
\newtheorem{lem}[theorem]{Lemma}

\newtheorem{cor}[theorem]{Corollary}

\theoremstyle{definition}

\newtheorem{example}[theorem]{Example}

\theoremstyle{remark}
\newtheorem{remark}[theorem]{Remark}

\numberwithin{equation}{section}

\begin{document}

\newcommand{\spacing}[1]{\renewcommand{\baselinestretch}{#1}\large\normalsize}
\spacing{1.14}
\title[Hermitian and K\"ahler Structures of Two Types on Lie Groups]{Hermitian and K\"ahler Structures of Two Types on Lie Groups with Two-Dimensional Derived Algebras}

\author {Hamid Reza Salimi Moghaddam}

\address{Department of Pure Mathematics,
Faculty of  Mathematics and Statistics,
University of Isfahan,
Isfahan, 81746-73441-Iran.\\
Scopus Author ID: 26534920800, ORCID Id:0000-0001-6112-4259}
\email{salimi.moghaddam@gmail.com and hr.salimi@sci.ui.ac.ir}

\date{\today}

\begin{abstract}
    We study left-invariant Hermitian structures on Lie groups whose derived algebra is two-dimensional. We consider two natural types of almost Hermitian structures, according to the position of the derived algebra with respect to the almost complex structure. For each type, we obtain explicit necessary and sufficient conditions for integrability. In particular, for the second type, integrability is equivalent to the abelianity of the complex structure. We then characterize the K\"ahler structures of both types and show that, in each case, the K\"ahler condition imposes strong restrictions on the Lie bracket. Moreover, we compute the corresponding Bismut connections and torsion forms explicitly. As applications, we construct examples of K\"ahler, WKT, and SKT structures in dimensions four, six, and  in all even dimensions greater than six. In particular, we obtain a family of abelian SKT structures on $\mathfrak r_2\oplus\mathfrak h_3\oplus\mathbb R^{2n-5}$ for $n\geq3$.



\end{abstract}

\maketitle


\section{Introduction}\label{Introduction}
An automorphism $J$ of the tangent bundle of a smooth manifold $M$ satisfying $J^2 = -\mathrm{Id}_{TM}$ is called an almost complex structure. The existence of such a structure implies that the dimension of $M$ is even. For any almost complex structure $J$, the Nijenhuis tensor $N_J$ is defined by
\begin{equation}\label{Nijenhuis tensor}
   N_J(X,Y) = [X,Y] + J\big([JX,Y] + [X,JY]\big) - [JX,JY].
\end{equation}
If $N_J = 0$, then $J$ is said to be integrable and is called a complex structure \cite{KoNo}.

Now, let $(M,g)$ be a Riemannian manifold equipped with an almost complex structure $J$. The pair $(J,g)$ is called an almost Hermitian structure if the metric $g$ is compatible with $J$ in the sense that (see \cite{Ivanov-Petkov,KoNo})
\begin{equation}\label{Hermitian equation}
    g(JX,JY) = g(X,Y) \quad \forall X,Y \in \mathfrak{X}(M).
\end{equation}
If $J$ is integrable, then $(J,g)$ is a Hermitian structure. Every Hermitian structure $(J,g)$ is associated with a fundamental $2$-form $\omega$, the K\"ahler form, defined by $\omega(X,Y) = g(JX,Y)$. The Hermitian structure is called a K\"ahler structure if the K\"ahler form is closed ($d\omega = 0$). An equivalent characterization is that the Levi-Civita connection $\nabla$ of $g$ preserves $J$, i.e., $\nabla J = 0$ \cite{KoNo}. K\"ahler structures play a significant role in theoretical physics \cite{Alvarez-Gaume-Freedman,Rocek,Zumino}.

For any Hermitian structure $(J,g)$ on a smooth manifold $M$, there exists a unique connection $\nabla^B$, known as the Bismut connection (or Strominger connection), satisfying $\nabla^B J = \nabla^B g = 0$ \cite{Andrada-Barberis,Bismut,Strominger}. The torsion tensor $T^B$ of $\nabla^B$ defines a $3$-form $c$ by:
\begin{equation}\label{3-form c}
    c(X,Y,Z) = g(X, T^B(Y,Z)), \quad \forall X,Y,Z \in \mathfrak{X}(M).
\end{equation}
This $3$-form is related to the K\"ahler form by the identity (see \cite{Andrada-Barberis}):
\begin{equation}\label{relation c and omega}
    c(X,Y,Z) = d\omega(JX,JY,JZ), \quad \forall X, Y, Z \in \mathfrak{X}(M).
\end{equation}
Using this relation and the Levi-Civita connection, the Bismut connection can be expressed explicitly as
\begin{equation}\label{Bismut-Levi-Civita equation}
    g(\nabla^B_X Y, Z) = g(\nabla_X Y, Z) + \frac{1}{2}c(X,Y,Z) \quad \forall X,Y,Z \in \mathfrak{X}(M).
\end{equation}

While some manifolds do not admit K\"ahler structures \cite{Latorre-Ugarte}, a weaker condition can be satisfied by certain Hermitian structures. A Hermitian structure is called strong K\"ahler with torsion (SKT) if the associated $3$-form is closed ($dc = 0$). If $c$ is not closed, the structure is referred to as weak (WKT) \cite{Andrada-Barberis}.

Hermitian structures on Lie groups provide an important class of examples in complex and differential geometry. In the left-invariant setting, many geometric questions can be reduced to algebraic conditions on the corresponding Lie algebra. This makes Lie groups particularly suitable for the explicit study of Hermitian, K\"ahler, and Hermitian structures with torsion. Let $G$ be a real Lie group with Lie algebra $\mathfrak{g}$ and derived algebra $\mathfrak{g}'$. An almost complex structure $J$ on $G$ is called left-invariant if it satisfies $J = T l_x \circ J \circ T l_{x^{-1}}$ for all $x \in G$. Furthermore, there is a one-to-one correspondence between inner products $\langle \cdot, \cdot \rangle$ on $\mathfrak{g}$ and left-invariant Riemannian metrics on $G$. The characterization of the two-step nilpotent Lie groups admitting a left-invariant complex structure is given by Barberis in \cite{Barberis2025}. In this article, we consider Lie groups with two-dimensional derived algebras, which are solvable Lie groups whose derived algebra is two-dimensional and abelian.

The dimension of $\frak{g}'$ provides a natural measure of the algebraic complexity of a solvable Lie algebra. In particular, the case $\dim\frak{g}'=1$ leads to a relatively simple description of left-invariant Riemannian and Hermitian structures. The author previously studied Hermitian and hyper-Hermitian structures in this setting \cite{Salimi Moghaddam1}. The purpose of the present paper is to investigate the next case, $\dim\frak{g}'=2$.
This case already exhibits substantially richer behavior and requires additional algebraic data.

A general almost Hermitian structure in this setting may have a nontrivial component of $J(\frak{g}')$ both in $\frak{g}'$ and in $\Gamma$, the orthogonal subspace to $\frak{g}'$. We do not attempt here to classify all such structures. Instead, we focus on two natural classes determined by the position of the derived algebra with respect to the almost complex structure. These two classes are distinguished by the interaction between the almost complex structure and the derived algebra and, importantly, admit an explicit description in terms of the algebraic data. The remaining mixed case, in which $J(\frak{g}')$ has nonzero components both in $\frak{g}'$ and in its orthogonal complement, is not considered in the present work.

Our first objective is to determine when these two types of almost Hermitian structures are integrable.We next study the K\"ahler condition within these two classes. Using the Bismut connection, we obtain explicit necessary and sufficient conditions for a Hermitian structure of Type I or Type II to be K\"ahler. In addition to characterizing the K\"ahler structures, we compute the Bismut connection and its torsion explicitly for both types. These formulas allow us to investigate Hermitian structures with torsion directly at the Lie algebra level. In particular, we provide examples of K\"ahler, WKT, and SKT structures. The examples include structures in dimensions four and six, as well as families in all even dimensions greater than six. Among them, we obtain a family of abelian SKT structures of Type II on $\frak{r}_2\oplus\frak{h}_3\oplus\Bbb{R}^{2n-5}$, for $n\geq3$.

The paper is organized as follows. In Section 2, we describe the algebraic structure of Lie groups with two-dimensional derived algebras and characterize Hermitian structures of Types I and II. Section 3 is devoted to K\"ahler structures and the corresponding Bismut connections. In Section 4, we present explicit examples illustrating the results, including K\"ahler, WKT, and SKT structures.


\section{Hermitian structures}\label{Hermitian structures section}
This section studies Hermitian structures on Lie groups with two-dimensional derived algebras. We begin with a brief description of such Lie groups equipped with left-invariant Riemannian metrics (for further details, see \cite{Salimi Moghaddam2}). We denote the set of $n$-dimensional Lie groups with two-dimensional derived algebras by $\mathrm{Lie}(n,2)$, where $n \geq 3$.

Let $G \in \mathrm{Lie}(n,2)$ with Lie algebra $\mathfrak{g}$. Then $G$ is solvable, and its derived algebra $\mathfrak{g}'$ is a two-dimensional abelian Lie algebra. Let $\langle \cdot, \cdot \rangle$ be an inner product on $\mathfrak{g}$, and let $g$ be the corresponding left-invariant Riemannian metric. Suppose $\{\mathsf{e}_1, \mathsf{e}_2\}$ is an orthonormal basis of $\mathfrak{g}'$, and let $\Gamma$ be the $(n-2)$-dimensional subspace of $\mathfrak{g}$ orthogonal to $\mathfrak{g}'$. There exist unique vectors $\mathsf{a}_1, \mathsf{a}_2, \mathsf{b}_1, \mathsf{b}_2 \in \Gamma$ such that:
\begin{align*}
  [u, \mathsf{e}_1] &= \langle \mathsf{a}_1, u \rangle \mathsf{e}_1 + \langle \mathsf{a}_2, u \rangle \mathsf{e}_2, \\
  [u, \mathsf{e}_2] &= \langle \mathsf{b}_1, u \rangle \mathsf{e}_1 + \langle \mathsf{b}_2, u \rangle \mathsf{e}_2, \quad \forall u \in \Gamma.
\end{align*}
Furthermore, there exist skew-adjoint linear maps $f_1, f_2 : \Gamma \to \Gamma$ such that for all $u, v \in \Gamma$:
\begin{equation*}
    [u, v] = \langle f_1(u), v \rangle \mathsf{e}_1 + \langle f_2(u), v \rangle \mathsf{e}_2.
\end{equation*}

By applying the Jacobi identity and the Cauchy-Schwarz inequality, one can show that $\mathsf{b}_1$ is a scalar multiple of $\mathsf{a}_2$. We do not use this relation to see the symmetry of the formulas.

As shown in \cite{Salimi Moghaddam2}, the Levi-Civita connection of $(G, g)$ is given by:
\begin{align*}
    & \nabla_{\mathsf{e}_1} \mathsf{e}_1 = \mathsf{a}_1, \ \ \ \nabla_{\mathsf{e}_1} \mathsf{e}_2 = \tfrac{1}{2}(\mathsf{a}_2 + \mathsf{b}_1), \ \ \ \nabla_{\mathsf{e}_1} u = -\langle \mathsf{a}_1, u \rangle \mathsf{e}_1 - \tfrac{1}{2}\left( \langle \mathsf{b}_1 + \mathsf{a}_2, u \rangle \mathsf{e}_2 + f_1(u) \right), \\
    & \nabla_{\mathsf{e}_2} \mathsf{e}_1 = \tfrac{1}{2}(\mathsf{a}_2 + \mathsf{b}_1), \ \ \ \nabla_{\mathsf{e}_2} \mathsf{e}_2 = \mathsf{b}_2, \ \ \ \nabla_{\mathsf{e}_2} u = -\langle \mathsf{b}_2, u \rangle \mathsf{e}_2 - \tfrac{1}{2}\left( \langle \mathsf{b}_1 + \mathsf{a}_2, u \rangle \mathsf{e}_1 + f_2(u) \right), \\
    & \nabla_v \mathsf{e}_1 = \tfrac{1}{2}\left( \langle \mathsf{a}_2 - \mathsf{b}_1, v \rangle \mathsf{e}_2 - f_1(v) \right), \ \ \ \nabla_v \mathsf{e}_2 = \tfrac{1}{2}\left( \langle \mathsf{b}_1 - \mathsf{a}_2, v \rangle \mathsf{e}_1 - f_2(v) \right), \\
    & \nabla_v u = \tfrac{1}{2}\left( \langle f_1(v), u \rangle \mathsf{e}_1 + \langle f_2(v), u \rangle \mathsf{e}_2 \right),
\end{align*}
for all $u, v \in \Gamma$. If $\{u, v\}$ is an orthonormal set in $\Gamma$, the sectional curvatures are:
\begin{align*}
    K(\mathsf{e}_1, \mathsf{e}_2) &= \tfrac{1}{4} \|\mathsf{a}_2 + \mathsf{b}_1\|^2 - \langle \mathsf{a}_1, \mathsf{b}_2 \rangle, \\
    K(\mathsf{e}_1, u) &= \tfrac{1}{4} \left( \langle \mathsf{b}_1, u \rangle^2 - 3 \langle \mathsf{a}_2, u \rangle^2 + \|f_1(u)\|^2 \right) - \langle \mathsf{a}_1, u \rangle^2 - \tfrac{1}{2} \langle \mathsf{a}_2, u \rangle \langle \mathsf{b}_1, u \rangle, \\
    K(\mathsf{e}_2, u) &= \tfrac{1}{4} \left( \langle \mathsf{a}_2, u \rangle^2 - 3 \langle \mathsf{b}_1, u \rangle^2 + \|f_2(u)\|^2 \right) - \langle \mathsf{b}_2, u \rangle^2 - \tfrac{1}{2} \langle \mathsf{a}_2, u \rangle \langle \mathsf{b}_1, u \rangle, \\
    K(u, v) &= -\tfrac{3}{4} \left( \langle u, f_1(v) \rangle^2 + \langle u, f_2(v) \rangle^2 \right).
\end{align*}

Now, let $(J, g)$ be an almost Hermitian structure on $G \in \mathrm{Lie}(n,2)$. Since $(J, g)$ is almost Hermitian, $\langle Jx, x \rangle = 0$ for all $x \in \mathfrak{g}$. Therefore, $J\mathsf{e}_1 = \lambda u + \mu \mathsf{e}_2$ for some $u \in \Gamma$ and $\lambda, \mu \in \mathbb{R}$ with $\lambda^2 + \mu^2 = 1$.

\begin{example}
    Let $\mathfrak{g} = \mathbb{R}^4$ be equipped with the Euclidean inner product $\langle \cdot, \cdot \rangle$ and the natural orthonormal basis $\{e_1, e_2, e_3, e_4\}$, with non-zero Lie brackets:
    \begin{equation*}
        [e_3, e_1] = e_1, \quad [e_3, e_4] = e_2.
    \end{equation*}
    Then $G \in \mathrm{Lie}(4,2)$, $\mathsf{e}_1=e_1$, $\mathsf{e}_2=e_2$, and the following almost complex structure $J$ defines an almost Hermitian structure:
    \begin{align*}
        J\mathsf{e}_1 &= \tfrac{\sqrt{2}}{2}(e_3 - \mathsf{e}_2), &
        J\mathsf{e}_2 &= \tfrac{\sqrt{2}}{2}(\mathsf{e}_1 + e_4), \\
        Je_3 &= \tfrac{\sqrt{2}}{2}(e_4 - \mathsf{e}_1), &
        Je_4 &= \tfrac{\sqrt{2}}{2}(-\mathsf{e}_2 - e_3).
    \end{align*}
\end{example}

In this paper, we do not consider almost Hermitian structures with both $\lambda \neq 0$ and $\mu \neq 0$ (as in the example above). Instead, we focus on the following two special cases.

Let $G \in \mathrm{Lie}(n,2)$ be equipped with an arbitrary left-invariant Riemannian metric $g$. We consider the following types of almost complex structures $J$ such that $(J, g)$ is an almost Hermitian structure.

\begin{description}
  \item[Type I] $J(\frak{g}')=\frak{g}'$,
  \item[Type II] $J(\frak{g}')\subseteq\Gamma$.
\end{description}
\begin{lem}\label{A SO(2)}
An almost Hermitian structure $J$ is of type I if and only if $J\mathsf{e}_1 =\pm\mathsf{e}_2$.
\end{lem}
\begin{proof}
One side is trivial, so we prove the other side. Suppose that $J(\frak{g}')=\frak{g}'$. Since $J$ is an almost Hermitian structure, $A:=J|_{\frak{g}'}:\frak{g}'\rightarrow\frak{g}'$ is also almost Hermitian, which means that $A\in O(2)$ and therefore $A^{-1}=A^T$. On the other hand, the relation $A^2=-I$ implies $A^T=-A$. Thus, the representation of $A$ with respect to the orthonormal basis $\{\mathsf{e}_1, \mathsf{e}_2\}$ is of the form
\begin{equation*}
    A=\left(
  \begin{array}{cc}
    0 & a \\
    -a & 0 \\
  \end{array}
\right).
\end{equation*}
Now, using $A^2=-I$, we obtain $a^2=1$, so $a=\pm1$. In particular $A\in SO(2)$.
\end{proof}
\begin{remark}\label{Remark 1-1}
If $J$ is an almost Hermitian structure of type I such that $J\mathsf{e}_1=-\mathsf{e}_2$, then by changing the orientation we can take $J\mathsf{e}_1=\mathsf{e}_2$. Therefore, in the remainder of this article, we consider Type I almost Hermitian structures only of the form $J\mathsf{e}_1=\mathsf{e}_2$.
\end{remark}

\begin{remark}\label{Remark 1-2}
If $J$ is an almost Hermitian structure of type II, then there exist $u_0^1, u_0^2\in\Gamma$ such that $J\mathsf{e}_1 = u_0^1$ and $J\mathsf{e}_2 = u_0^2$. Also, $Ju \in \Gamma \cap (\mathrm{span}\{u_0^1, u_0^2\})^\perp$ for all $u \in \Gamma \cap (\mathrm{span}\{u_0^1, u_0^2\})^\perp$. Thus, for $n = 2k$, there exist $u_1, \dots, u_{k-2} \in \Gamma \cap (\mathrm{span}\{u_0^1, u_0^2\})^\perp$ such that
      \begin{equation*}
        \{\mathsf{e}_1, \mathsf{e}_2, u_0^1, u_0^2, u_1, \dots, u_{k-2}, v_1 = Ju_1, \dots, v_{k-2} = Ju_{k-2}\}
      \end{equation*}
      is an orthonormal basis for $\mathfrak{g}$.
\end{remark}

We now give necessary and sufficient conditions for almost Hermitian structures of Types I and II to be Hermitian. Thus, the integrability of a Type I almost Hermitian structure can be completely expressed in terms of the four vectors $\mathsf{a}_1, \mathsf{a}_2, \mathsf{b}_1, \mathsf{b}_2$ and the two skew-adjoint endomorphisms $f_1, f_2$.

\begin{theorem}\label{theorem 1}
    An almost Hermitian structure of Type I is Hermitian if and only if the following equations hold:
    \begin{align*}
        \mathsf{b}_2 - \mathsf{a}_1 &= J(\mathsf{a}_2 + \mathsf{b}_1), \\
        Jf_1 - f_1J &= Jf_2J + f_2.
    \end{align*}
\end{theorem}
\begin{proof}
    A direct computation shows that for any $u, v \in \Gamma$:
    \begin{align*}
        N_J(\mathsf{e}_1, \mathsf{e}_2) &= 0, \\
        N_J(\mathsf{e}_1, u) &= \langle \mathsf{b}_2 - \mathsf{a}_1 - J(\mathsf{a}_2 + \mathsf{b}_1), u \rangle \mathsf{e}_1 + \langle J(\mathsf{a}_1 - \mathsf{b}_2) - \mathsf{a}_2 - \mathsf{b}_1, u \rangle \mathsf{e}_2, \\
        N_J(\mathsf{e}_2, u) &= \langle J(\mathsf{a}_1 - \mathsf{b}_2) - \mathsf{b}_1 - \mathsf{a}_2, u \rangle \mathsf{e}_1 + \langle \mathsf{a}_1 - \mathsf{b}_2 + J(\mathsf{b}_1 + \mathsf{a}_2), u \rangle \mathsf{e}_2, \\
        N_J(u, v) &= \langle f_1(u) - f_2(Ju) + J(f_2(u) + f_1(Ju)), v \rangle \mathsf{e}_1 \\
        &\quad + \langle f_2(u) + f_1(Ju) + J(f_2(Ju) - f_1(u)), v \rangle \mathsf{e}_2,
    \end{align*}
    which completes the proof.
\end{proof}

\begin{remark}\label{remark 2}
    For an almost Hermitian structure of Type I, if $(J, g)$ is Hermitian, then the first equation of Theorem~\ref{theorem 1} implies $\langle \mathsf{a}_1, \mathsf{a}_2 \rangle - \langle \mathsf{b}_1, \mathsf{b}_2 \rangle = \langle \mathsf{a}_2, \mathsf{b}_2 \rangle - \langle \mathsf{a}_1, \mathsf{b}_1 \rangle$. This relation can be used as a criterion to show that $(J, g)$ is not Hermitian.
\end{remark}

An almost complex structure $J$ is called abelian if $[Jx, Jy] = [x, y]$ for all $x, y \in \mathfrak{g}$.
\begin{remark}\label{remark 3}
    Any abelian almost Hermitian structure is Hermitian \cite{Barberis-Dotti}. It is easily verified that an almost Hermitian structure of Type I is abelian if and only if $J\mathsf{a}_1 = \mathsf{b}_1$, $J\mathsf{a}_2 = \mathsf{b}_2$, $Jf_1 = f_1J$, and $Jf_2 = f_2J$.
\end{remark}

\begin{remark}\label{Remark 1-3}
    If $\dim{\frak{g}}=4$, then the equation $Jf_1-f_1J=Jf_2J+f_2$ holds automatically. In this case, with respect to a suitable basis for $\Gamma$, the linear map $J|_\Gamma:\Gamma\rightarrow\Gamma$ has the following representation:
    \begin{equation*}
    J|_\Gamma=\left(
  \begin{array}{cc}
    0 & -1 \\
    1 & 0 \\
  \end{array}
\right).
\end{equation*}
On the other hand, the space of all skew-symmetric linear maps from $\Gamma$ to itself is one-dimensional. Since $f_1$ and $f_2$ are skew-symmetric, they are real multiples of $J|_\Gamma$. Therefore, $Jf_1-f_1J=Jf_2J+f_2=0$. Hence, if $\dim{\frak{g}}=4$, then an almost Hermitian structure $J$ of type I is Hermitian if and only if $\mathsf{b}_2 - \mathsf{a}_1 = J(\mathsf{a}_2 + \mathsf{b}_1)$.
\end{remark}
For Type II structures, we obtain the stronger phenomenon that integrability is equivalent to the abelianity of the complex structure.
\begin{theorem}\label{theorem 4}
    Let $(J, g)$ be a left-invariant almost Hermitian structure of Type II. Then the following statements are equivalent:
    \begin{enumerate}
        \item $(J, g)$ is a Hermitian structure.
        \item $J$ is an abelian structure.
        \item The following equations hold:
        \begin{align*}
            \tilde{\mathsf{a}}_1 &= Jf_1(u_0^1), &
            \tilde{\mathsf{a}}_2 &= Jf_2(u_0^1), \\
            \tilde{\mathsf{b}}_1 &= Jf_1(u_0^2), &
            \tilde{\mathsf{b}}_2 &= Jf_2(u_0^2), \\
            \langle \mathsf{b}_1, u_0^1 \rangle &= \langle \mathsf{a}_1, u_0^2 \rangle, &
            \langle \mathsf{b}_2, u_0^1 \rangle &= \langle \mathsf{a}_2, u_0^2 \rangle, \\
            \langle f_1(Ju) - Jf_1(u), v \rangle &= 0, &
            \langle f_2(Ju) - Jf_2(u), v \rangle &= 0, \quad \forall u, v \in \Gamma \cap (\mathrm{span}\{u_0^1, u_0^2\})^\perp,
        \end{align*}
        where for any $x \in \Gamma$, $\tilde{x}$ denotes the orthogonal projection of $x$ onto $(\mathrm{span}\{u_0^1, u_0^2\})^\perp$.
    \end{enumerate}
\end{theorem}
\begin{proof}
    The structure $(J, g)$ is Hermitian if and only if $N_J(X, Y) = 0$ for all $X, Y \in \mathfrak{g}$. A direct computation for $X, Y \in \{\mathsf{e}_1, \mathsf{e}_2, u_0^1, u_0^2, u, v\}$ shows that $N_J(X, Y) = 0$ is equivalent to the following ten equations:
    \begin{align}
        \langle f_1(u_0^1), u_0^2 \rangle &= 0, &
        \langle f_2(u_0^1), u_0^2 \rangle &= 0, \nonumber \\
        \langle f_1(u_0^1) + J\mathsf{a}_1, u \rangle &= 0, &
        \langle f_2(u_0^1) + J\mathsf{a}_2, u \rangle &= 0, \nonumber \\
        \langle f_1(u_0^2) + J\mathsf{b}_1, u \rangle &= 0, &
        \langle f_2(u_0^2) + J\mathsf{b}_2, u \rangle &= 0, \label{eq:ten_equations} \\
        \langle \mathsf{b}_1, u_0^1 \rangle &= \langle \mathsf{a}_1, u_0^2 \rangle, &
        \langle \mathsf{b}_2, u_0^1 \rangle &= \langle \mathsf{a}_2, u_0^2 \rangle, \nonumber \\
        \langle f_1(Ju) - Jf_1(u), v \rangle &= 0, &
        \langle f_2(Ju) - Jf_2(u), v \rangle &= 0. \nonumber
    \end{align}
    Another direct computation shows that $J$ is abelian if and only if the same ten equations hold. Thus, statements (1) and (2) are equivalent.

    We now prove that equations \eqref{eq:ten_equations} are equivalent to statement (3). Note that the maps $f_1$ and $f_2$ are skew-adjoint, and
    \begin{align*}
        \langle J\mathsf{a}_1, u_0^1 \rangle &= \langle J\mathsf{a}_1, u_0^2 \rangle = \langle J\mathsf{a}_2, u_0^1 \rangle = \langle J\mathsf{a}_2, u_0^2 \rangle = 0, \\
        \langle J\mathsf{b}_1, u_0^1 \rangle &= \langle J\mathsf{b}_1, u_0^2 \rangle = \langle J\mathsf{b}_2, u_0^1 \rangle = \langle J\mathsf{b}_2, u_0^2 \rangle = 0.
    \end{align*}
    Therefore, the first six equations of \eqref{eq:ten_equations} hold if and only if:
    \begin{align*}
        f_1(u_0^1) &= -J\mathsf{a}_1 - \langle \mathsf{a}_1, u_0^1 \rangle \mathsf{e}_1 - \langle \mathsf{a}_1, u_0^2 \rangle \mathsf{e}_2, \\
        f_2(u_0^1) &= -J\mathsf{a}_2 - \langle \mathsf{a}_2, u_0^1 \rangle \mathsf{e}_1 - \langle \mathsf{a}_2, u_0^2 \rangle \mathsf{e}_2, \\
        f_1(u_0^2) &= -J\mathsf{b}_1 - \langle \mathsf{b}_1, u_0^1 \rangle \mathsf{e}_1 - \langle \mathsf{b}_1, u_0^2 \rangle \mathsf{e}_2, \\
        f_2(u_0^2) &= -J\mathsf{b}_2 - \langle \mathsf{b}_2, u_0^1 \rangle \mathsf{e}_1 - \langle \mathsf{b}_2, u_0^2 \rangle \mathsf{e}_2,
    \end{align*}
    which completes the proof.
\end{proof}


\section{K\"ahler structures}\label{Kahler structures section}
This section characterizes K\"ahler structures of Types I and II. In the proofs of Theorems \ref{theorem 5} and \ref{theorem 7}, although it might be somewhat simpler to compute the K\"ahler  form and find the required closure conditions, we preferred to use a more involved method in order to obtain the Bismut connections simultaneously.

\begin{theorem}\label{theorem 5}
    A Hermitian structure of Type I is K\"ahler if and only if $f_1 = f_2 = 0$ and $\mathsf{b}_2 = -\mathsf{a}_1$.
\end{theorem}
\begin{proof}
    For every $u, v, w \in \Gamma$, formula \eqref{relation c and omega} gives:
    \begin{align*}
        c(\mathsf{e}_1, \mathsf{e}_2, u) &= -\langle \mathsf{a}_1 + \mathsf{b}_2, Ju \rangle, &
        c(\mathsf{e}_1, u, v) &= \langle Ju, f_1(Jv) \rangle, \\
        c(\mathsf{e}_2, u, v) &= \langle Ju, f_2(Jv) \rangle, &
        c(u, v, w) &= 0.
    \end{align*}
    Using equation \eqref{Bismut-Levi-Civita equation}, the Bismut connection is given by:
    \begin{align*}
        & \nabla^B_{\mathsf{e}_1} \mathsf{e}_1 = \mathsf{a}_1, \ \ \ \ \ \nabla^B_{\mathsf{e}_1} \mathsf{e}_2 = J\mathsf{a}_1, \\
        & \nabla^B_{\mathsf{e}_1} u = -\langle \mathsf{a}_1, u \rangle \mathsf{e}_1 - \langle J\mathsf{a}_1, u \rangle \mathsf{e}_2 - f_1(u) + \tfrac{1}{2}(f_2J - Jf_2)(u), \\
        & \nabla^B_{\mathsf{e}_2} \mathsf{e}_1 = -J\mathsf{b}_2, \ \ \ \ \ \nabla^B_{\mathsf{e}_2} \mathsf{e}_2 = \mathsf{b}_2, \\
        & \nabla^B_{\mathsf{e}_2} u = \langle J\mathsf{b}_2, u \rangle \mathsf{e}_1 - \langle \mathsf{b}_2, u \rangle \mathsf{e}_2 - f_2(u) + \tfrac{1}{2}(Jf_1 - f_1J)(u), \\
        & \nabla^B_v \mathsf{e}_1 = \langle J\mathsf{a}_1 - \mathsf{b}_1, v \rangle \mathsf{e}_2 - \tfrac{1}{2}(f_2J - Jf_2)(v), \\
        & \nabla^B_v \mathsf{e}_2 = -\langle \mathsf{a}_2 + J\mathsf{b}_2, v \rangle \mathsf{e}_1 - \tfrac{1}{2}(Jf_1 - f_1J)(v), \\
        & \nabla^B_v u = \tfrac{1}{2} \langle (f_2J - Jf_2)(v), u \rangle \mathsf{e}_1 + \tfrac{1}{2} \langle (Jf_1 - f_1J)(v), u \rangle \mathsf{e}_2.
    \end{align*}
    The torsion tensor of the Bismut connection is then:
    \begin{align*}
        T^B(\mathsf{e}_1, \mathsf{e}_2) &= J(\mathsf{a}_1 + \mathsf{b}_2), \\
        T^B(\mathsf{e}_1, u) &= -\langle J(\mathsf{a}_1 + \mathsf{b}_2), u \rangle \mathsf{e}_2 + Jf_1(Ju), \\
        T^B(\mathsf{e}_2, u) &= \langle J(\mathsf{a}_1 + \mathsf{b}_2), u \rangle \mathsf{e}_1 + Jf_2(Ju), \\
        T^B(u, v) &= \langle Jf_1(Ju), v \rangle \mathsf{e}_1 + \langle Jf_2(Ju), v \rangle \mathsf{e}_2.
    \end{align*}
    Since $(J, g)$ is K\"ahler if and only if $T^B = 0$, it follows that the structure is K\"ahler precisely when $f_1 = f_2 = 0$ and $\mathsf{b}_2 = -\mathsf{a}_1$.
\end{proof}

\begin{cor}\label{corollary 6}
    There are no abelian K\"ahler structures of Type I.
\end{cor}
\begin{proof}
Using the Jacobi identity for the triple $(\mathsf{e}_1,u,v)$ together with $f_1 = f_2 = 0$ and $\mathsf{b}_2 = -\mathsf{a}_1$, we obtain $\langle \mathsf{a}_2, u \rangle \langle \mathsf{a}_1, v \rangle = \langle \mathsf{a}_1, u \rangle \langle \mathsf{a}_2, v \rangle,$ for all $u,v\in\Gamma$. Thus $\mathsf{a}_1$ and $\mathsf{a}_2$ are linearly dependent. Since $\mathsf{a}_2=J\mathsf{a}_1$ and $J$ is orthogonal with $J^2=-I$, we have
$\langle \mathsf{a}_1,J\mathsf{a}_1\rangle=0$. Consequently, $\mathsf{a}_1=0$. It follows that $\mathsf{a}_2=\mathsf{b}_1=\mathsf{b}_2=0$. Together with $f_1=f_2=0$, all the Lie brackets vanish, so $\frak{g}'=0$, contrary to the assumption that $\dim\frak{g}'=2$. Therefore, there are no abelian K\"ahler structures of Type I.
\end{proof}

\begin{theorem}\label{theorem 7}
    A Hermitian structure of Type II is K\"ahler if and only if $f_1 = f_2 = 0$ and $\mathsf{b}_1 = \mathsf{a}_2$.
\end{theorem}
\begin{proof}
    Using relation \eqref{relation c and omega}, for any $u, v, w \in \Gamma \cap (\mathrm{span}\{u_0^1, u_0^2\})^\perp$, we compute:
    \begin{align*}
        c(\mathsf{e}_1, \mathsf{e}_2, u_0^1) &= \langle \mathsf{a}_1, u_0^2 \rangle - \langle \mathsf{a}_2, u_0^1 \rangle, &
        c(\mathsf{e}_1, \mathsf{e}_2, u_0^2) &= \langle \mathsf{b}_1, u_0^2 \rangle - \langle \mathsf{b}_2, u_0^1 \rangle, \\
        c(\mathsf{e}_1, \mathsf{e}_2, u) &= \langle f_2(u_0^1) - f_1(u_0^2), Ju \rangle, &
        c(\mathsf{e}_1, u_0^1, u_0^2) &= 0, \\
        c(\mathsf{e}_1, u_0^1, u) &= \langle J\mathsf{a}_1, u \rangle, &
        c(\mathsf{e}_1, u_0^2, u) &= \langle J\mathsf{b}_1, u \rangle, \\
        c(\mathsf{e}_1, u, v) &= \langle Ju, f_1(Jv) \rangle, &
        c(\mathsf{e}_2, u_0^1, u_0^2) &= 0, \\
        c(\mathsf{e}_2, u_0^1, u) &= \langle J\mathsf{a}_2, u \rangle, &
        c(\mathsf{e}_2, u_0^2, u) &= \langle J\mathsf{b}_2, u \rangle, \\
        c(\mathsf{e}_2, u, v) &= \langle Ju, f_2(Jv) \rangle, &
        c(u_0^1, u_0^2, u) &= c(u_0^1, u, v) = c(u_0^2, u, v) = c(u, v, w) = 0.
    \end{align*}
    Using these and formula \eqref{Bismut-Levi-Civita equation}, the Bismut connection is given by:
    \begin{align*}
        \nabla^B_{\mathsf{e}_1} \mathsf{e}_1 &= \mathsf{a}_1, &
        \nabla^B_{\mathsf{e}_1} \mathsf{e}_2 &= \mathsf{b}_1, &
        \nabla^B_{\mathsf{e}_1} u_0^1 &= J\mathsf{a}_1, &
        \nabla^B_{\mathsf{e}_1} u_0^2 &= J\mathsf{b}_1, \\
        \nabla^B_{\mathsf{e}_2} \mathsf{e}_1 &= \mathsf{a}_2, &
        \nabla^B_{\mathsf{e}_2} \mathsf{e}_2 &= \mathsf{b}_2, &
        \nabla^B_{\mathsf{e}_2} u_0^1 &= J\mathsf{a}_2, &
        \nabla^B_{\mathsf{e}_2} u_0^2 &= J\mathsf{b}_2, \\
        \nabla^B_{\mathsf{e}_1} u &= -\langle \mathsf{a}_1, u \rangle \mathsf{e}_1 - \langle \mathsf{b}_1, u \rangle \mathsf{e}_2 - f_1(u), \\
        \nabla^B_{\mathsf{e}_2} u &= -\langle \mathsf{a}_2, u \rangle \mathsf{e}_1 - \langle \mathsf{b}_2, u \rangle \mathsf{e}_2 - f_2(u), \\
        \nabla^B_{u_0^1} x &= \nabla^B_{u_0^2} x = \nabla^B_u x = 0.
    \end{align*}
    The torsion tensor of $\nabla^B$ is then:
    \begin{align*}
        T^B(\mathsf{e}_1, \mathsf{e}_2) &= \mathsf{b}_1 - \mathsf{a}_2, \\
        T^B(\mathsf{e}_1, u_0^1) &= \langle \mathsf{a}_1, u_0^1 \rangle \mathsf{e}_1 + \langle \mathsf{a}_2, u_0^1 \rangle \mathsf{e}_2 + J\mathsf{a}_1, \\
        T^B(\mathsf{e}_1, u_0^2) &= \langle \mathsf{a}_1, u_0^2 \rangle \mathsf{e}_1 + \langle \mathsf{a}_2, u_0^2 \rangle \mathsf{e}_2 + J\mathsf{b}_1, \\
        T^B(\mathsf{e}_1, u) &= \langle \mathsf{a}_2 - \mathsf{b}_1, u \rangle \mathsf{e}_2 - f_1(u), \\
        T^B(\mathsf{e}_2, u_0^1) &= \langle \mathsf{b}_1, u_0^1 \rangle \mathsf{e}_1 + \langle \mathsf{b}_2, u_0^1 \rangle \mathsf{e}_2 + J\mathsf{a}_2, \\
        T^B(\mathsf{e}_2, u_0^2) &= \langle \mathsf{b}_1, u_0^2 \rangle \mathsf{e}_1 + \langle \mathsf{b}_2, u_0^2 \rangle \mathsf{e}_2 + J\mathsf{b}_2, \\
        T^B(\mathsf{e}_2, u) &= \langle \mathsf{b}_1 - \mathsf{a}_2, u \rangle \mathsf{e}_1- f_2(u), \\
        T^B(u_0^1, u_0^2) &= 0, \\
        T^B(u_0^1, u) &= -\langle f_1(u_0^1), u \rangle \mathsf{e}_1 - \langle f_2(u_0^1), u \rangle \mathsf{e}_2, \\
        T^B(u_0^2, u) &= -\langle f_1(u_0^2), u \rangle \mathsf{e}_1 - \langle f_2(u_0^2), u \rangle \mathsf{e}_2, \\
        T^B(u, v) &= -\langle f_1(u), v \rangle \mathsf{e}_1 - \langle f_2(u), v \rangle \mathsf{e}_2.
    \end{align*}
    It is clear that $T^B = 0$ if and only if $f_1 = f_2 = 0$ and $\mathsf{b}_1 = \mathsf{a}_2$.
\end{proof}


\section{Examples}\label{Examples}
In this section, we give some examples to illustrate the simplicity of the results of the present paper, although some of them are familiar to specialists.
\begin{example}\label{example 8}
    (Four-dimensional abelian WKT structure of Type I)\\
    Let $\mathfrak{g}$ be the Lie algebra $A_{4,12}$ from \cite{Patera-Sharp-Winternitz}, with non-zero Lie brackets:
    \begin{equation*}
        [e_1, e_3] = e_1, \quad [e_2, e_3] = e_2, \quad [e_1, e_4] = -e_2, \quad [e_2, e_4] = e_1.
    \end{equation*}
    Equip $\mathfrak{g}$ with the inner product $\langle \cdot, \cdot \rangle$ for which $\{e_1, e_2, e_3, e_4\}$ is an orthonormal basis. Then $\mathfrak{g} = A_{4,12}$ is a Lie algebra with a two-dimensional derived algebra. In our notation, we have $\Gamma = \mathrm{span}\{e_3, e_4\}$ and:
    \begin{align*}
        \mathsf{e}_1 &= e_1, & \mathsf{e}_2 &= e_2, & \mathsf{a}_1 &= -e_3, & \mathsf{a}_2 &= e_4, \\
        \mathsf{b}_1 &= -e_4, & \mathsf{b}_2 &= -e_3, & f_1 &= 0, & f_2 &= 0.
    \end{align*}
    Consider the almost complex structure $J$ defined by $J\mathsf{e}_1 = \mathsf{e}_2, Je_3 = e_4$.

    By Remark~\ref{remark 3}, $(J, g)$ is an abelian Hermitian structure of Type I. Theorem~\ref{theorem 5} (or Corollary~\ref{corollary 6}) shows that $(J, g)$ is not K\"ahler. Moreover,
    \begin{equation*}
        dc(e_1, e_2, e_3, e_4) = -4 \neq 0,
    \end{equation*}
    so $(J, g)$ is a four-dimensional abelian WKT structure of Type I.

    The Bismut connection of $(J, g)$ is given by:
    \begin{align*}
        \nabla^B_{\mathsf{e}_1} \mathsf{e}_1 &= -e_3, &
        \nabla^B_{\mathsf{e}_1} \mathsf{e}_2 &= -e_4, &
        \nabla^B_{\mathsf{e}_1} e_3 &= \mathsf{e}_1, &
        \nabla^B_{\mathsf{e}_1} e_4 &= \mathsf{e}_2, \\
        \nabla^B_{\mathsf{e}_2} \mathsf{e}_1 &= e_4, &
        \nabla^B_{\mathsf{e}_2} \mathsf{e}_2 &= -e_3, &
        \nabla^B_{\mathsf{e}_2} e_3 &= \mathsf{e}_2, &
        \nabla^B_{\mathsf{e}_2} e_4 &= -\mathsf{e}_1, \\
        \nabla^B_{e_3} x &= 0, &
        \nabla^B_{e_4} x &= 0, & & \forall x \in \mathfrak{g}.
    \end{align*}
\end{example}

\begin{example}\label{example 10}
    (Four-dimensional abelian WKT structure of Type II)\\
    Let the Lie algebra $\mathfrak{g}$ and the inner product $\langle \cdot, \cdot \rangle$ be the same as in Example~\ref{example 8}. Let $J$ be the abelian complex structure of Type II defined by $Je_1 = e_4, Je_2 = e_3$.

    Then we have $\mathsf{e}_1 = e_1$, $\mathsf{e}_2 = e_2$, $\mathsf{a}_1 = -e_3$, $\mathsf{a}_2 = e_4$, $\mathsf{b}_1 = -e_4$, $\mathsf{b}_2 = -e_3$, $f_1 = f_2 = 0$, $u_0^1 = e_4$, and $u_0^2 = e_3$. Theorem~\ref{theorem 4} shows that $(J, g)$ is a Hermitian structure. However, since $\mathsf{b}_1 = -e_4 \neq e_4 = \mathsf{a}_2$, Theorem~\ref{theorem 7} implies it is not K\"ahler. Furthermore, from the proof of Theorem~\ref{theorem 7}, we find
    \begin{equation*}
        dc(e_1, e_2, e_3, e_4) = 4 \neq 0,
    \end{equation*}
    so $(J, g)$ is a four-dimensional abelian WKT structure of Type II. The Bismut connection for this Hermitian structure is the same as in Example~\ref{example 8}.

    This example, together with Example~\ref{example 8}, shows that a Lie algebra equipped with an inner product can simultaneously admit two different WKT structures of Types I and II with the same Bismut connection.
\end{example}
In \cite{Zhang-Zheng}, it is shown that flat Hermitian Lie algebras are K\"ahler. In the following we give an example of a non-flat Hermitian Lie algebra which is K\"ahler.
\begin{example}\label{example 12}
    (Four-dimensional abelian K\"ahler structure of Type II)\\
    Let $\mathfrak{g}$ be the four-dimensional Lie algebra with non-zero Lie brackets:
    \begin{equation*}
        [e_3, e_1] = e_1, \quad [e_4, e_2] = e_2,
    \end{equation*}
    equipped with the inner product for which $\{e_1, e_2, e_3, e_4\}$ is an orthonormal basis. Then we have:
    \begin{align*}
        \mathsf{e}_1 &= e_1, & \mathsf{e}_2 &= e_2, & \mathsf{a}_1 &= e_3, & \mathsf{b}_2 &= e_4, \\
        \mathsf{b}_1 &= 0, & \mathsf{a}_2 &= 0, & u_0^1 &= e_3, & u_0^2 &= e_4, & f_1 = f_2 &= 0.
    \end{align*}
    Consider the almost complex structure $J$ defined by $J\mathsf{e}_1 = e_3$ and $J\mathsf{e}_2 = e_4$. One can verify that $J$ is abelian. Therefore, by Theorem~\ref{theorem 4}, $(J, g)$ is a Hermitian structure, and by Theorem~\ref{theorem 7}, it is K\"ahler.

    The Bismut connection of this structure is:
    \begin{align*}
        \nabla^B_{\mathsf{e}_1} \mathsf{e}_1 &= e_3, &
        \nabla^B_{\mathsf{e}_1} \mathsf{e}_2 &= 0, &
        \nabla^B_{\mathsf{e}_1} e_3 &= -\mathsf{e}_1, &
        \nabla^B_{\mathsf{e}_1} e_4 &= 0, \\
        \nabla^B_{\mathsf{e}_2} \mathsf{e}_1 &= 0, &
        \nabla^B_{\mathsf{e}_2} \mathsf{e}_2 &= e_4, &
        \nabla^B_{\mathsf{e}_2} e_3 &= 0, &
        \nabla^B_{\mathsf{e}_2} e_4 &= -\mathsf{e}_2, \\
        \nabla^B_{e_3} x &= 0, &
        \nabla^B_{e_4} x &= 0, & & \forall x \in \mathfrak{g}.
    \end{align*}
    In this case we have $K(\mathsf{e}_1, u_0^1) = -1$, which shows that this four-dimensional abelian K\"ahler structure of Type II is not flat.
\end{example}

\begin{example}\label{example 13}
    (Six-dimensional abelian SKT structure of Type I)\\
Now we give an example on a nilpotent Lie group of dimension six. For more details about almost Hermitian structures and strong K\"ahler structure with torsion in dimension six, see \cite{Abbena-Garbiero-Salamon} and \cite{Fino-Parton-Salamon}. Suppose that $\mathfrak{g}$ is the Lie algebra $A_{6,4}$ from \cite{Patera-Sharp-Winternitz}, with non-zero Lie brackets:
    \begin{equation*}
        [e_5, e_6] = e_1, \quad [e_5, e_3] = e_2, \quad [e_6, e_4] = e_2.
    \end{equation*}
    Equip $\mathfrak{g}$ with the inner product $\langle \cdot, \cdot \rangle$ for which $\{e_1, \ldots, e_6\}$ is an orthonormal basis. In the notation of this paper, we have $\Gamma = \mathrm{span}\{e_3, \ldots ,e_6\}$ and:
    \begin{equation*}
        \mathsf{e}_1 = e_1,  \mathsf{e}_2 = e_2, \mathsf{a}_1=\mathsf{a}_2=\mathsf{b}_1=\mathsf{b}_2=0.
    \end{equation*}
Also, the matrix representations of $f_1$ and $f_2$, with respect to the basis $\{e_3, \ldots ,e_6\}$ are as follows:
\begin{equation*}
    f_1=\left(
       \begin{array}{cccc}
         0 & 0 & 0 & 0 \\
         0 & 0 & 0 & 0 \\
         0 & 0 & 0 & -1 \\
         0 & 0 & 1 & 0 \\
       \end{array}
     \right), \ \ \ f_2=\left(
                      \begin{array}{cccc}
                        0 & 0 & 1 & 0 \\
                        0 & 0 & 0 & 1 \\
                        -1 & 0 & 0 & 0 \\
                        0 & -1 & 0 & 0 \\
                      \end{array}
                    \right).
\end{equation*}
Now, consider the almost Hermitian structure $J$ of type I defined by:
    \begin{align*}
        J\mathsf{e}_1 &= \mathsf{e}_2, & J\mathsf{e}_3 &= \mathsf{e}_4, & J e_5 &= e_6.
    \end{align*}
$J$ is abelian, so using Remark~\ref{remark 3}, $(J, g)$ is a Hermitian structure of Type I. Corollary~\ref{corollary 6} shows that $(J, g)$ is not a K\"ahler structure. Using $\mathsf{a}_1=\mathsf{a}_2=\mathsf{b}_1=\mathsf{b}_2=0$, a direct computation shows that $dc=0$. Therefore, $(J,g)$ is a six-dimensional abelian SKT structure of Type I.

For its Bismut connection, we have:
    \begin{align*}
        \nabla^B_{\mathsf{e}_1} \mathsf{e}_1 &=\nabla^B_{\mathsf{e}_1} \mathsf{e}_2=\nabla^B_{\mathsf{e}_2} \mathsf{e}_1=\nabla^B_{\mathsf{e}_2} \mathsf{e}_2=\nabla^B_{\mathsf{e}_1} e_3=\nabla^B_{\mathsf{e}_1} e_4 =0, \\
        \nabla^B_{\mathsf{e}_1} e_5 &= -e_6, \ \ \ \ \ \nabla^B_{\mathsf{e}_1} e_6 = e_5, \ \ \ \ \ \nabla^B_{\mathsf{e}_2} e_3 = e_5, \\
        \nabla^B_{\mathsf{e}_2} e_4 &= e_6, \ \ \ \ \ \nabla^B_{\mathsf{e}_2} e_5 = -e_3, \ \ \ \ \ \nabla^B_{\mathsf{e}_2} e_6 = -e_4,\\
        \nabla^B_x\mathsf{e}_1 &= \nabla^B_x\mathsf{e}_2=\nabla^B_xy = 0, \ \ \ \ \ \forall x,y \in\Gamma.
    \end{align*}
\end{example}
In the following two examples, for any $n\geq4$, we provide a $2n$-dimensional abelian WKT structure of Type I, and for any $n\geq3$, a $2n$-dimensional abelian SKT structure of Type II.
\begin{example}\label{example 14}
    ($2n$-dimensional abelian WKT structure of Type I ($n\geq4$))\\
Let $n=m+1$, where $m\geq3$, and let $\frak{g}=\mathrm{span}\{\mathsf{e}_1,\mathsf{e}_2,x_1,\ldots, x_m, y_1,\ldots, y_m\}$ be a Lie algebra with the Lie brackets:
\begin{equation*}
    [x_1,x_2]=\mathsf{e}_2, \ \ \ \ [x_i,y_i]=\mathsf{e}_1, \ \ i=1,\ldots,m,
\end{equation*}
equipped with the inner product $\langle \cdot, \cdot \rangle$ for which the above basis is orthonormal. It follows directly from the bracket relations that $\frak{g}$ is a two-step nilpotent Lie algebra with a two-dimensional derived algebra, and $\mathsf{a}_1=\mathsf{a}_2=\mathsf{b}_1=\mathsf{b}_2=0$. Also, the matrix representations of $f_1$ and $f_2$ with respect to the basis $\{x_1,\ldots, x_m, y_1,\ldots, y_m\}$ are as follows:
\[
f_1 =
\left(
\begin{array}{c|c}
0_m & -I_m \\ \hline
I_m & 0_m
\end{array}
\right), \ \ \ \
f_2 =
\left(
\begin{array}{c|c}
\begin{matrix}
0 & -1 \\
1 & 0
\end{matrix} & 0_{2\times(2m-2)} \\ \hline
0_{(2m-2)\times2} & 0_{(2m-2)\times(2m-2)}
\end{array}
\right).
\]

Now, we consider the type I almost Hermitian structure $J$ on $\frak{g}$ defined by:
\begin{eqnarray*}
  && J\mathsf{e}_1=\mathsf{e}_2, \ J\mathsf{e}_2=-\mathsf{e}_1, \ Jx_1=x_2, \ Jx_2=-x_1, \\ \\
  && Jy_1=y_2, \ Jy_2=-y_1, \ Jx_i=y_i, \ Jy_i=-x_i, \ \ \  i=3,\ldots,m.
\end{eqnarray*}
We see that $J$ is abelian, and using Remark \ref{remark 3} it is Hermitian. Since $f_1\neq0$ (and $f_2\neq0$), Theorem \ref{theorem 5} implies that the Hermitian structure $J$ is not K\"ahler. One can also use the equation $d\omega(\mathsf{e}_1,x_1,x_2)=1$ to see that $J$ is not K\"ahler. Moreover, $dc(x_1,y_1,x_2,y_2)=2$, so the 3-form $c$ is not closed. Thus, the Hermitian structure $(J,\langle \cdot, \cdot \rangle)$ is a WKT structure.
\end{example}

The following family is particularly useful because it provides, in every even dimension $2n\ge6$, an explicit abelian SKT structure of Type II.
\begin{example}\label{example 15}
    ($2n$-dimensional abelian SKT structure of Type II ($n\geq3$))\\
Let $\frak{r}_2=\mathrm{span}\{\mathsf{e}_1,u_0^1\}$ be the non-abelian two-dimensional Lie algebra with the Lie bracket $[\mathsf{e}_1,u_0^1]=\mathsf{e}_1$. Suppose that $\frak{h}_3$ denotes the Heisenberg Lie algebra $\frak{h}_3=\mathrm{span}\{\mathsf{e}_2,u_1,v_1\}$ with the Lie bracket $[u_1,v_1]=\mathsf{e}_2$. Also, suppose that $\Bbb{R}^{2n-5}=\mathrm{span}\{u_0^2,u_2,\ldots,u_{n-2},v_2,\ldots,v_{n-2}\}$ with $n\geq4$ (if $n=3$ set $\Bbb{R}=\mathrm{span}\{u_0^2\}$) is the trivial abelian Lie algebra. Now let $\frak{g}=\frak{r}_2\oplus\frak{h}_3\oplus\Bbb{R}^{2n-5}$, with $n\geq3$, be the $2n$-dimensional Lie algebra equipped with an inner product such that the basis $\{\mathsf{e}_1,\mathsf{e}_2,u_0^1,u_0^2,u_1,\ldots, u_{n-2}, v_1,\ldots, v_{n-2}\}$ is an orthonormal basis. In this case, we have $\mathsf{a}_1 = -u_0^1$, $\mathsf{a}_2=\mathsf{b}_1=\mathsf{b}_2=0$, $f_1=0$, and $f_2(u_1)=v_1$. We define an almost complex structure $J$ by
\[
\begin{aligned}
J\mathsf{e}_1&=u_0^1, &
Ju_0^1&=-\mathsf{e}_1,\\
J\mathsf{e}_2&=u_0^2, &
Ju_0^2&=-\mathsf{e}_2,\\
Ju_i&=v_i, &
Jv_i&=-u_i,
\qquad i=1,\ldots,n-2.
\end{aligned}
\]

$J$ is abelian, so $(J,\langle \cdot, \cdot \rangle)$ is Hermitian. On the other hand, $f_2\neq0$ together with Theorem \ref{theorem 7} shows that the abelian Hermitian structure $(J,\langle \cdot, \cdot \rangle)$
is not K\"ahler. \\
Let $\{\bar{\mathsf{e}}_1,\bar{\mathsf{e}}_2,\bar{u}_0^1,\bar{u}_0^2,\bar{u}_1,\ldots, \bar{u}_{n-2}, \bar{v}_1,\ldots, \bar{v}_{n-2}\}$ be the dual basis. The structure equations are
\[
d\bar{\mathsf{e}}_1=-\bar{\mathsf{e}}_1\wedge \bar{u}_0^1,
\qquad
d\bar{\mathsf{e}}_2=-\bar{u}_1\wedge \bar{v}_1,
\]
while all other differentials of basis one-forms vanish.
The  K\"ahler $2$-form
\[
\omega(X,Y)=g(JX,Y)
\]
is therefore
\[
\omega
=
\bar{\mathsf{e}}_1\wedge \bar{u}_0^1
+
\bar{\mathsf{e}}_2\wedge \bar{u}_0^2
+
\sum_{i=1}^{n-2}\bar{u}_i\wedge \bar{v}_i.
\]
Consequently,
\[
d\omega
=
-\bar{u}_0^2\wedge \bar{u}_1\wedge \bar{v}_1.
\]
Now for the $3$-form $c$, we have:
\[
c
=
-\bar{\mathsf{e}}_2\wedge \bar{u}_1\wedge \bar{v}_1.
\]
Since
\[
d\bar{\mathsf{e}}_2=-\bar{u}_1\wedge \bar{v}_1,
\qquad
d\bar{u}_1=d\bar{v}_1=0,
\]
we obtain
\[
dc
=
-d\bar{\mathsf{e}}_2\wedge \bar{u}_1\wedge \bar{v}_1
=
(\bar{u}_1\wedge \bar{v}_1)\wedge \bar{u}_1\wedge \bar{v}_1
=
0.
\]
Hence, the corresponding Hermitian structure is SKT.
\end{example}


\section*{Declarations}

{\textbf{Conflict of interest:} On behalf of all authors, the corresponding author states that there is no conflict of interest.\\

{\textbf{Ethical approval:} Not applicable.\\

{\textbf{Competing interests:} The author has no relevant financial or non-financial interests to disclose.\\

{\textbf{Authors' contributions:}} Not applicable.\\

{\textbf{Funding:}} There is not any financial support. \\

{\textbf{Availability of data and materials:}} Data sharing not applicable to this article as no datasets were generated or analysed during the current study.\\


\end{document}